\documentclass[11pt, a4paper,reqno]{amsart}

\usepackage{amssymb}

\newcommand{\natdot}{{\raisebox{-2pt}{\scalebox{0.4}{\(\bullet\)}}}}

\usepackage{enumitem}
\renewlist{enumerate}{enumerate}{3}
\setlist[enumerate,1]{label=\rm(\arabic*)}
\setlist[enumerate,2]{label=\rm(\alph*)}
\setlist[enumerate,3]{label=\rm(\roman*)}

\usepackage[PS]{diagrams}
\usepackage{xcolor}
\usepackage{pgf}

\newcommand{\iso}{\cong}
\newcommand{\ie}{\textit{i.e.}}

\newcommand{\resp}{\textit{resp.}}

\newcommand{\ignore}[1]{\relax}
\newcommand{\bN}{\mathbb{N}}
\newcommand{\bZ}{\mathbb{Z}}

\newcommand{\bR}{\mathbb{R}}

\newcommand{\nbd}{\nobreakdash}
\newcommand{\id}{\ensuremath{\mathrm{id}}}
\newcommand{\tensor}{\otimes}

\newcommand{\shift}[2]{{}^{(#1)}#2}
\newcommand{\syst}[2]{\textbf{Syst}_{#1}\text{{\rm -}}#2}
\newcommand{\Sy}[1]{#1 \text{{\rm -}} \text{\rm\bf syst}}

\numberwithin{equation}{section}

\newtheorem*{theorem*}{Theorem}

\newtheorem{theorem}[equation]{Theorem}
\newtheorem{corollary}[equation]{Corollary}
\newtheorem{proposition}[equation]{Proposition}
\newtheorem{lemma}[equation]{Lemma}

\theoremstyle{definition}
\newtheorem{definition}[equation]{Definition}
\newtheorem{remark}[equation]{Remark}

\newcommand{\noqed}{\renewcommand{\qedsymbol}{}\ignorespaces}

\let\oldtocsubsection=\tocsubsection
\renewcommand{\tocsubsection}[2]{\hspace{3.5em}\(\cdot\)~\oldtocsubsection{#1}{#2}}

\usepackage[hypertexnames=false,colorlinks=false,pdfborderstyle={/S/U/W 1}]{hyperref}

\newcommand{\cA}{\mathcal{A}}
\newcommand{\cB}{\mathcal{B}}
\newcommand{\cF}{\mathcal{F}}
\newcommand{\cP}{\mathcal{P}}
\newcommand{\inv}{^{-1}}
\newcommand{\smm}[1]{\left( \smallmatrix #1 \endsmallmatrix \right)}
\newcommand{\mat}[1]{\left( \begin{matrix} #1 \end{matrix} \right)}

\newcommand{\lt}{\mathfrak{LT}}
\newcommand{\idem}{\mathrm{Idem}\,}
\newcommand{\na}[1]{\noalign{\noindent#1}}
\newcommand{\pow}[1]{\mathbb{P}(#1)^{\mathrm{fin}}}

\newcommand{\dirlim}{\lim\limits_{\rightarrow}}

\begin{document}

\title{Triangular objects and systematic $K$-theory}

\date{\today}

\author{Thomas H\"uttemann}

\address{Thomas H\"uttemann\\ Queen's University Belfast\\ School of
  Mathematics and Physics\\ Pure Mathematics Research Centre\\ Belfast
  BT7~1NN\\ UK}

\email{t.huettemann@qub.ac.uk}

\urladdr{http://www.qub.ac.uk/puremaths/Staff/Thomas Huettemann/}

\subjclass[2010]{Primary 19D50; Secondary 18E05}

\author{Zuhong Zhang}

\address{Zhang Zuhong\\ Beijing Institute Of Technology\\ School of
  Mathematics\\ 5 South Zhongguancun Street, Haidian District\\ 100081
  Beijing\\ China}

\email{zuhong@hotmail.com}

\begin{abstract}
  We investigate modules over ``systematic'' rings.  Such rings are
  ``almost graded'' and have appeared under various names in the
  literature; they are special cases of the $G$-systems of
  \textsc{Grzeszczuk}. We analyse their $K$-theory in the presence of
  conditions on the support, and explain how this generalises and
  unifies calculations of graded and filtered $K$-theory scattered in
  the literature. Our treatment makes systematic use of the formalism
  of idempotent completion and a theory of triangular objects in
  additive categories, leading to elementary and transparent proofs
  throughout.
\end{abstract}

\maketitle

\tableofcontents

\section*{Introduction}

The aim of this note is to provide a unified treatment of several
results in algebraic $K$-theory, comparing ``graded'' (\resp,
``filtered'') $K$-theory of graded (\resp, filtered) rings with the
usual algebraic $K$-theory of the subring in degree (\resp, filtration
degree)~$0$. A typical case is the following result:

\begin{theorem*}[\textsc{Quillen} {\cite[p.~107,
    Proposition]{MR0338129}}]
  \label{thm:Q-orig}
  Let $B$ a positively graded ring (that is, a $\bZ$-graded ring with
  $B_{k} = \{0\}$ for $k<0$). Then there is a $\bZ[x,x^{-1}]$-linear
  isomorphism
  \begin{displaymath}
    \bZ[x,x^{-1}] \tensor_{\bZ} K_{n} (B_{0}) \iso K^{\mathrm{gr}}_{n}
    (B) \ , \quad x^{n} \tensor P \mapsto P \tensor_{B_{0}} \shift n B
    \ ,
  \end{displaymath}
  with $K^{\mathrm{gr}}_{n} (B)$ denoting the algebraic $K$\nbd-theory
  of the category of finitely generated $\bZ$\nbd-graded projective
  $B$-modules, and $\shift n B$ denoting the graded module with
  $\shift n B_{k} = B_{k-n}$.
\end{theorem*}

The proof given in {\it loc.cit.\/}~is short but subtle, involving
certain non-canonical isomorphisms between various modules; it is
quite surprising that, after passing to suitable quotients, all
constructions become functorial and hence induce maps on
\textsc{Quillen} $K$-groups. Similar complications can be found, in
more explicit form, in~\cite{MR3095326} and~\cite{MR3169434}.

In this paper we propose an alternative approach which, while still
based on the additivity theorem (or a version of ``characteristic
filtrations''), is more explicit and transparent. In Part~1 we develop
an axiomatic setup for applying the additivity theorem to triangular
objects in additive categories. In Part~2 we introduce systematic
rings and modules, a notion that generalises and unifies both graded
and filtered algebra at once. In Part~3 we study the algebraic
$K$-theory of systematic rings; our computations specialise to various
known calculations of graded and filtered $K$-theory scattered in the
literature. We end the paper with remarks on the algebraic $K$-theory
of affine toric schemes.

\section*{Acknowledgements}

Work on this paper began during a research visit of the first author
to Beijing Institute of Technology in Summer 2014. Their hospitality
and financial support is gratefully acknowledged.

\part{General theory of lower triangular categories}

\section{The idempotent completion of an additive category}
\label{sec:idem}

Let $\cA$ be an additive category. The {\it idempotent completion},
or {\it \textsc{Karoubi} envelope}, of~$\cA$ is the additive
category~$\idem \cA$ defined as follows: Objects are pairs $(A,p)$
with $p \colon A \rTo A$ an idempotent morphism in~$\cA$ (that is,
objects are ``projections'' in~$\cA$); a morphism $f \colon (A,p) \rTo
(B,q)$ is a morphism $f \colon A \rTo B$ in~$\cA$ such that $qfp = f$
(or, equivalently, $fp=f=qf$); in particular, $\idem\cA\big((A,p),
(A', p') \big)$ is a subset of $\cA(A,A')$. Identity morphisms are
given by $\id_{(A,p)} = p$, and composition of morphisms is inherited
from composition in~$\cA$.

The functor $A \mapsto (A, \id_{A})$ and $f \mapsto f$ is an embedding
of~$\cA$ as a full subcategory of~$\idem\cA$. It is an equivalence
if and only if all idempotents in~$\cA$ split (\ie, if and only if for
every idempotent morphism $p \colon A \rTo A$ there exist morphisms $r
\colon A \rTo A'$ and $s \colon A' \rTo A$ such that $r \circ s =
\id_{A'}$ and $p = s \circ r$).

A standard example is the category~$\cA$ with objects the based
finitely generated free $R$-modules $R^{n}$, $n \geq 0$, and morphisms
the $R$-linear maps, for some fixed unital ring~$R$. The
category~$\cA$ is equivalent to the category of all finitely generated
free $R$-modules, and its idempotent completion $\idem\cA$ is
equivalent to the category of finitely generated projective
$R$-modules.

A functor $\Phi \colon \cA \rTo \cB$ between additive categories
induces a functor
\begin{displaymath}
  \idem \Phi = \hat \Phi \colon \idem \cA \rTo \idem \cB \ , \quad
  (F,p) \mapsto \big( \Phi(F), \Phi(p) \big) \ ;
\end{displaymath}
for a morphism $f \colon (F,p) \rTo (G,q)$ we have
\begin{displaymath}  \hat \Phi(f) = \Phi(f) \colon
  \big( \Phi(F), \Phi(p) \big) \rTo \big( \Phi(G), \Phi(q) \big) \ 
  \in \idem \cB \ ,
\end{displaymath}
as $f=qfp$ implies $\Phi(f) = \Phi(q) \Phi(f) \Phi(p)$ by
functoriality.

\begin{lemma}
  \label{lem:additive}
  If $\Phi \colon \cA \rTo \cB$ is additive then so is the induced
  functor $\hat \Phi \colon \idem\cA \rTo \idem\cB$.
\end{lemma}

\begin{proof}
  By definition of additive functor, the functor~$\Phi$ yields group
  homomorphisms $\cA(A,A') \rTo \cB \big( \Phi(A), \Phi(A') \big)$.
  Hence so does~$\hat \Phi$ as its effect on underlying morphisms is
  that of~$\Phi$; that is, $\hat\Phi$ is additive.
\end{proof}

Suppose that $\Phi,\Psi \colon \cA \rTo \cB$ are additive
functors, and that $\tau \colon \Phi \rTo^\natdot \Psi$ is a natural
transformation. Then $\hat \tau$, defined by
\begin{displaymath}
  \hat \tau_{(A,p)} = \Psi(p) \circ \tau_{A} \colon \big(
  \Phi(A), \Phi(p) \big) \rTo \big( \Psi(A), \Psi(p) \big) \ ,
\end{displaymath}
is a natural transformation of functors $\hat \Phi \rTo^\natdot \hat
\Psi$. Indeed, we have
\begin{align*}
  \Psi(p) \circ \hat \tau_{(A,p)} \circ \Phi(p) & = \Psi(p) \circ \big(
  \Psi(p) \circ \tau_{A}) \big) \circ \Phi(p) && \text{(definition of
    \(\hat\tau\))}\\
  & = \Psi(p) \circ \big( \tau_{A} \circ \Phi(p) \big) \circ \Phi(p)
  && \text{(naturality of \(\tau\))}\\
  & = \Psi(p) \circ \big( \tau_{A} \circ \Phi(p) \big) &&
  \text{(\(\Phi(p)\) idempotent)} \\
  & = \Psi(p) \circ \big( \Psi(p) \circ \tau_{A}) && \text{(naturality
    of \(\tau\))}\\ 
  &= \Psi(p) \circ \tau_{A} && \text{(\(\Psi(p)\) idempotent)} \\
  &= \hat \tau_{(A,p)} && \text{(definition of \(\hat\tau\))}
\end{align*}
so that $\hat \tau_{(A,p)}$ is a morphism
in~$\idem \cB$. To verify naturality, fix a morphism $a \colon (A,p) \rTo (A',
p')$ in~$\idem \cA$. Then we compute
\begin{align*}
  \hat \tau_{(A', p')} \circ \hat \Phi(a) &= \Psi(p') \circ \tau_{A'} \circ
  \Phi(a) && \text{(definition of \(\hat \tau\))}\\
  &= \Psi(p') \circ \Psi(a) \circ \tau_{A} && \text{(naturality of~\(\tau\))}
  \\
  &= \Psi(a) \circ \Psi(p) \circ \tau_{A} && \text{(as \(p'a=a=ap\))} \\
  &= \hat\Psi(a) \circ \hat \tau_{(A,p)} && \text{(definition of
    \(\hat \tau\))} \ .
\end{align*}

\begin{lemma}
  \label{lem:equivalent}
  If $\tau$ is a natural isomorphism then $\hat \tau$ is a natural
  isomorphism as well. Equivalent additive categories thus have
  equivalent idempotent completions.
\end{lemma}

\begin{proof}
  As $\tau\inv$ is a natural transformation $\Psi \rTo^\natdot \Phi$ the
  construction above yields a natural transformation
  $\widehat{\tau\inv} \colon \hat \Psi \rTo^\natdot \hat \Phi$. We claim
  that this is the inverse of~$\hat\tau$. Indeed, for an object $(A,p)
  \in \idem\cA$ we calculate
  \begin{align*}
    \big( \widehat{\tau\inv} \circ \hat \tau \big)_{(A,p)} &=
    \widehat{\tau\inv}_{(A,p)} \circ \hat \tau_{(A,p)} \\
    &= \big( \Phi(p) \circ \tau\inv_{A} \big) \circ \big( \Psi(p)
    \circ
    \tau_{A} \big) && \text{(definition of \(\hat\tau\) and \(\widehat{\tau\inv}\))}\\
    &= \Phi(p) \circ \tau\inv_{A} \circ \tau_{A} \circ \Phi(p) &&
    \text{(naturality of~\(\tau\))}\\
    &= \Phi(p) && \text{(as \(p \circ p = p\))}\\
    &= \id_{\hat \Phi (A, p)}
  \end{align*}
  so that $\widehat{\tau\inv} \circ \hat \tau = \id_{\hat\Phi}$.  A
  similar calculation shows $\hat \tau \circ \widehat{\tau\inv} =
  \id_{\hat\Psi}$ as well.
\end{proof}

We will from now on drop the decoration ``$\,\hat\ \,$'' and
let~$\Phi$ denote both the original additive functor and the induced
functor~$\hat \Phi$ discussed above. --- We will make use of the
following fact, which can be verified by explicit calculation:

\begin{lemma}
  \label{lem:filtration}
  Idempotent completion is compatible with filtered colimits:
  \begin{displaymath}
    \idem \big(\dirlim \cA[S]\big) = \dirlim \big(\idem \cA[S]\big) \ ,
  \end{displaymath}
  where $S$ varies over a directed poset (or, more generally, a small
  filtered category) and $S \mapsto \cA[S]$ is a system of additive
  categories and additive functors. \qed
\end{lemma}

\section{Exact structures}

In this note we consider all additive categories as exact categories
with the split exact structure, that is, by declaring a sequence to be
exact if and only if it is split exact. By definition, a sequence $0
\rTo A \rTo B \rTo C \rTo 0$ is split exact if there exists an
isomorphism $\chi \colon B \rTo A \oplus C$ resulting in a commutative
ladder diagram
\begin{diagram}[small]
  0 & \rTo & A & \rTo & B & \rTo & C & \rTo & 0 \\
  && \dTo<{\id_{A}} && \dTo<{\chi} && \dTo<{\id_{C}} \\
  0 & \rTo & A & \rTo[l>=4em]^{\text{incl}} & A \oplus C &
  \rTo[l>=4em]^{\text{proj}} & C & \rTo & 0 & \ .
\end{diagram}

\section{Lower triangular categories}
\label{sec:lt_cat}

Let $\cA$ be an additive category, and let $\cA_{1}$ and~$\cA_{2}$ be
full additive subcategories. We define the lower triangular category
$\lt(\cA_{1}, \cA_{2})$ to be the category which has objects
\begin{displaymath}
  F = F_{1} \oplus F_{2} \ , \quad F_{j} \in \cA_{j}
\end{displaymath}
(the direct sum decomposition being part of the data), and has
morphisms the lower triangular matrices
\begin{displaymath}
  f = \mat{f_{11}\\f_{21}&f_{22}} \colon F_{1} \oplus F_{2} = F \rTo G
  = G_{1} \oplus G_{2}
\end{displaymath}
with $f_{ij} \in \cA(F_{j}, G_{i})$. The category $\lt(\cA_{1},
\cA_{2})$ comes with a faithful (but not full) forgetful functor
to~$\cA$, for which we do not introduce special notation. Perhaps more
importantly, $\lt(\cA_{1}, \cA_{2})$ is itself an additive category;
the direct sum of $F_{1} \oplus F_{2}$ and $G_{1} \oplus G_{2}$
in~$\lt(\cA_{1}, \cA_{2})$ is described by the following sum system:
\begin{displaymath}
  F_{1} \oplus F_{2} \pile{\rTo^{\mat{\smm{\id\\0}\\ 0 & \smm{\id\\0}}}
    \\ \lTo_{\mat{\smm{\id&0}\\ 0 & \smm{\id&0}}}}
  (F_{1} \oplus G_{1}) \oplus
  (F_{2} \oplus G_{2}) \pile{\rTo^{\mat{\smm{0&\id}\\ 0 & \smm{0&\id}}}
    \\ \lTo_{\mat{\smm{0\\ \id}\\ 0 &
        \smm{0\\ \id}}}}
  G_{1} \oplus G_{2}
\end{displaymath}

For future reference we record a rather trivial calculation:

\begin{lemma}
  Suppose that the morphism
  \begin{displaymath}
    \mat{p_{11}\\p_{21}&p_{22}} \colon (A_{1}\oplus A_{2}) \rTo
    (A_{1}\oplus A_{2})
  \end{displaymath}
  in $\lt(\cA_{1}, \cA_{2})$ is idempotent. Then we have equalities
  \begin{subequations}
    \begin{align}
      p_{11}^{2} &= p_{11} \ , \label{eq:p11}\\
      p_{22}^{2} &= p_{22} \ ,\label{eq:p22}\\
      p_{21}p_{11} + p_{22}p_{21} &= p_{21} \ ; \label{eq:p21} \\
      \na{the latter implies, by multiplication with~\(p_{11}\) from
        the right and re-arranging,}
      p_{22}p_{21}p_{11} &= 0 \ . \label{eq:0}
    \end{align}
  \end{subequations}
  \qed
\end{lemma}

There are embeddings of $\cA_{1}$ and~$\cA_{2}$ as full subcategories
of~$\lt(\cA_{1},\cA_{2})$, given by sending an object $A$ to $A \oplus
0$ and $0 \oplus A$, respectively. These functors yield full
embeddings (often suppressed from the notation in the following)
\begin{subequations}
  \begin{align}
    \epsilon_{1} \colon \idem\cA_{1} &\rTo \idem\lt(\cA_{1},\cA_{2}) \
    , & (A_{1},
    p_{11}) &\mapsto \Big( A_{1} \oplus 0, \mat{p_{11}\\0&0}
    \Big) \label{eq:epsilon1} \\
    \na{and} %
    \epsilon_{2} \colon \idem\cA_{2} &\rTo \idem\lt(\cA_{1},\cA_{2}) \
    , & (A_{2}, p_{22}) & \mapsto \Big( 0 \oplus A_{2},
    \mat{0\\0&p_{22}} \Big) \ . \label{eq:epsilon2}
  \end{align}
\end{subequations}

\section{Subobject and quotient object functors}
\label{sec:functors}

We keep the notation from the previous section. There are
functors
\begin{gather*}
  S \colon \lt(\cA_{1}, \cA_{2}) \rTo \cA_{2} \ , \quad A_{1} \oplus
  A_{2} \mapsto A_{2} \ , \quad \mat{f_{11}\\f_{21}&f_{22}} \mapsto f_{22} \\
  \na{and} Q \colon \lt(\cA_{1}, \cA_{2}) \rTo \cA_{1} \ , \quad A_{1}
  \oplus A_{2} \mapsto A_{1} \ , \quad \mat{f_{11}\\f_{21}&f_{22}}
  \mapsto f_{11}
\end{gather*}
which induce functors on idempotent completions
\begin{align*}
  S \colon \big( A_{1} \oplus A_{2},
  \mat{\alpha_{11}\\\alpha_{21}&\alpha_{22}} \big) & \mapsto (A_{2},
  \alpha_{22}) \\
  \na{and} Q \colon \big( A_{1} \oplus A_{2},
  \mat{\alpha_{11}\\\alpha_{21}&\alpha_{22}} \big) & \mapsto (A_{1},
  \alpha_{11}) \ .
\end{align*}

\begin{lemma}
  \label{lemma:additive_functors}
  The functors
  \begin{gather*}
    S \colon \idem\lt(\cA_{1},\cA_{2}) \rTo \idem\cA_{2} \\
    \na{and}
    Q \colon \idem\lt(\cA_{1},\cA_{2}) \rTo \idem\cA_{1}
  \end{gather*}
  defined above are exact (that is, map split short exact sequences to
  split short exact sequences).
\end{lemma}

\begin{proof}
  The functors $S \colon \lt(\cA_{1}, \cA_{2}) \rTo \cA_{2}$ and $Q
  \colon \lt(\cA_{1}, \cA_{2}) \rTo \cA_{1}$ are additive, hence so
  are the induced functors after idempotent completion, by
  Lemma~\ref{lem:additive}. This is equivalent to the assertion under
  consideration.
\end{proof}

\begin{lemma}
\label{lem:ses_functors}
  There is a short exact sequence of functors
  \begin{displaymath}
    0 \rTo S \rTo \id_{\idem \lt(\cA_{1}, \cA_{2})} \rTo Q \rTo 0 \ .
  \end{displaymath}
  More precisely,
  \begin{enumerate}
  \item for every object $A = (A_{1}\oplus A_{2},
    \smm{p_{11}\\p_{21}&p_{22}})$ of~$\idem\lt(\cA_{1}, \cA_{2})$
    there is a sequence in~$\idem\cA$
    \begin{equation}
      \label{eq:ses}
      0 \rTo S(A) \rTo[l>=4em]_{\sigma}^{\smm{0\\p_{22}}} A
      \rTo_{\pi}[l>=4em]^{\smm{p_{11}&0}} Q(A) \rTo 0 \ ,
    \end{equation}
    and this sequence is split exact in the following way: There
    exists a morphism $\rho \colon A \rTo S(A)$ in $\idem\cA$
    such that the induced map
    \begin{displaymath}
      \mat{\pi\\ \rho}\colon A \rTo Q(A) \oplus S(A)
    \end{displaymath}
    is an isomorphism in $\idem\lt(\cA_{1}, \cA_{2})$;
  \item the sequence~\eqref{eq:ses} is natural in~$A$ with respect to
    morphisms in $\idem\lt(\cA_{1}, \cA_{2})$.
  \end{enumerate}
\end{lemma}

\begin{proof}
  Let $A = (A_{1}\oplus A_{2}, \smm{p_{11}\\p_{21}&p_{22}})$ be an
  object of~$\idem\lt(\cA_{1}, \cA_{2})$. Then we have the
  sequence~\eqref{eq:ses} of composable morphisms in~$\idem\cA$; explicitly:
  \begin{displaymath}
    0 \rTo (A_{2}, p_{22}) \rTo^{\smm{0\\p_{22}}}_{\sigma}
    \Big(A_{1}\oplus A_{2}, \mat{p_{11}\\p_{21}&p_{22}} \Big)
    \rTo^{\smm{p_{11}&0}}_{\pi} (A_{1}, p_{11}) \rTo 0 \ . 
  \end{displaymath}
  Clearly $\pi \circ \sigma = 0$, and note that as objects of
  $\idem\lt(\cA_{1}, \cA_{2})$ we have
  \begin{displaymath}
    (A_{1}, p_{11}) \oplus (A_{2}, p_{22}) = \Big(A_{1} \oplus A_{2},
    \mat{p_{11}\\&p_{22}}\Big) \ .
  \end{displaymath}
  We define $\rho \colon A \rTo S(A)$ to be the morphism in~$\idem\cA$
  \begin{displaymath}
    \rho = \mat{p_{22}p_{21}&p_{22}} \colon \Big(A_{1}\oplus
    A_{2}, \mat{p_{11}\\p_{21}&p_{22}} \Big) \rTo (A_{2}, p_{22}) \ ;
  \end{displaymath}
  we claim that $\rho$ and~$\pi$ yield a morphism in
  $\idem\lt(\cA_{1}, \cA_{2})$
  \begin{multline*}
    \mat{\pi\\ \rho} = \mat{p_{11}\\p_{22}p_{21}&p_{22}} \colon \\
    \Big(A_{1}\oplus A_{2}, \mat{p_{11}\\p_{21}&p_{22}} \Big) \rTo
    \Big(A_{1} \oplus A_{2}, \mat{p_{11}\\&p_{22}}\Big) \ .
  \end{multline*}
  Indeed, using the equations~\eqref{eq:p11} and~\eqref{eq:p22} we
  calculate
  \begin{displaymath}
    \mat{p_{11}\\&p_{22}} \cdot
    \mat{p_{11}\\p_{22}p_{21}&p_{22}} \cdot
    \mat{p_{11}\\p_{21}&p_{22}} =
    \mat{p_{11}\\p_{22}p_{21}p_{11}+p_{22}p_{21}&p_{22}} \ ,
  \end{displaymath}
  which coincides with the map~$\smm{\pi\\\rho}$ by~\eqref{eq:0} as
  required.

  Next, we define
  \begin{displaymath}
    M = \mat{p_{11}\\p_{21}p_{11}&p_{22}} \colon 
    \Big(A_{1} \oplus A_{2}, \mat{p_{11}\\&p_{22}}\Big) \rTo
    \Big(A_{1}\oplus A_{2}, \mat{p_{11}\\p_{21}&p_{22}} \Big) \ ;    
  \end{displaymath}
  this is a morphism in $\idem\lt(\cA_{1}, \cA_{2})$ as
  \begin{displaymath}
    \mat{p_{11}\\p_{21}&p_{22}} \cdot
    \mat{p_{11}\\p_{21}p_{11}&p_{22}} \cdot \mat{p_{11}\\&p_{22}} =
    \mat{p_{11}\\p_{21}p_{11} + p_{22}p_{21}p_{11}&p_{22}} = M \ ,
  \end{displaymath}
  using \eqref{eq:p11}, \eqref{eq:p22} and~\eqref{eq:0} again.

  Finally, we calculate
  \begin{displaymath}
    M \cdot \mat{\pi\\\rho} = \mat{p_{11}\\p_{21}p_{11}&p_{22}}
    \cdot \mat{p_{11}\\p_{22}p_{21}&p_{22}} = \mat{p_{11}\\p_{21}&p_{22}}
  \end{displaymath}
  (using~\eqref{eq:p21} for the (2,1)-entry of the last matrix) which
  is the identity map of $(A_{1} \oplus
  A_{2},\smm{p_{11}\\p_{21}&p_{22}})$. Similarly,
  \begin{displaymath}
    \mat{\pi\\\rho} \cdot M = \mat{p_{11}\\p_{22}p_{21}&p_{22}} \cdot
    \mat{p_{11}\\p_{21}p_{11}&p_{22}} = \mat{p_{11}\\&p_{22}}
  \end{displaymath}
  (using~\eqref{eq:0} for the (2,1)-entry of the last matrix) which is
  the identity map of $(A_{1} \oplus
  A_{2},\smm{p_{11}\\&p_{22}})$. This shows that $\smm{\pi\\\rho}$ is
  an isomorphism in $\idem\lt(\cA_{1},\cA_{2})$ with inverse~$M$, and
  finishes the proof of part~(1).

  It remains to verify naturality of the sequence~\eqref{eq:ses} with
  respect to morphisms in~$\idem\lt(\cA_{1},\cA_{2})$
    \begin{displaymath}
      \mat{f_{11}&\\ f_{21}&f_{22}} \colon \Big(A_{1} \oplus A_{2},
      \mat{p_{11}\\p_{21}&p_{22}} \Big)
      \rTo \Big(B_{1} \oplus B_{2}, \mat{q_{11}\\q_{21}&q_{22}} \Big) \ .
    \end{displaymath}
    Using the defining property of a morphism in the idempotent
    completion
    \begin{displaymath}
      \mat{f_{11}&\\ f_{21}&f_{22}} \cdot \mat{p_{11}\\p_{21}&p_{22}}
      = \mat{f_{11}&\\ f_{21}&f_{22}} = \mat{q_{11}\\q_{21}&q_{22}}
      \cdot \mat{f_{11}&\\ f_{21}&f_{22}} \ ,
    \end{displaymath}
    it is a routine calculation to verify commutativity of the diagram
    \begin{diagram}
            0 & \rTo & (A_{2}, p_{22}) &
            \rTo[l>=4em]^{\smm{0\\p_{22}}} & (A_{1} \oplus A_{2}, p) &
            \rTo[l>=4em]^{(p_{11}\ 0)} & (A_{1}, p_{11}) & \rTo & 0 \\
            && \dTo<{f_{22}} && \dTo<{\smm{f_{11}&\\ f_{21}&f_{22}}} &&
            \dTo<{f_{11}} \\
            0 & \rTo & (B_{2}, q_{22}) &
            \rTo^{\smm{0\\q_{22}}} & (B_{1} \oplus B_{2}, q) &
            \rTo^{(q_{11}\ 0)} & (B_{1}, q_{11}) & \rTo & 0 
    \end{diagram}
    which finishes the proof of part~(2).
\end{proof}

\begin{remark}
  The splitting map~$\rho$ from Lemma~\ref{lem:ses_functors}~(1) is
  {\it not\/} natural in~$A$.
\end{remark}

\section{Algebraic $K$-theory}
\label{sec:K}

\begin{lemma}
  \label{lem:K_split}
  There is an isomorphism of \textsc{Quillen} $K$-groups
  \begin{displaymath}
    (Q_{*}, S_{*}) \colon K_{n} \big(
    \idem\lt(\cA_{1},\cA_{2}) \big) \rTo^{\iso} K_{n} (\idem\cA_{1})
    \oplus K_{n} (\idem \cA_{2}) \ ,
  \end{displaymath}
  given by sending $A = \big( A_{1} \oplus A_{2},
  \smm{p_{11}\\p_{21}&p_{22}} \big)$ to $Q(A) = (A_{1},p_{11})$
  and $S(A) = (A_{2}, p_{22})$. The inverse
  $\epsilon_{1*}+\epsilon_{2*}$ is induced by the functor
  \begin{align*}
    \epsilon_{1} + \epsilon_{2} \colon \idem \cA_{1} \times \idem
    \cA_{2} & \rTo \idem \lt(\cA_{1},\cA_{2}) \ , \\
    \big( (A_{1}, p_{11}),\, (A_{2}, p_{22}) \big) & \mapsto \Big(
    A_{1} \oplus A_{2},\, \smm{p_{11}\\&p_{22}} \Big) \ .
  \end{align*}
\end{lemma}

\begin{proof}
  Using the embeddings of categories \eqref{eq:epsilon1}
  and~\eqref{eq:epsilon2}, we can re-phrase the conclusion of
  Lemma~\ref{lem:ses_functors}: There is a short exact sequence of
  endo-functors of $\idem\lt(\cA_{1}, \cA_{2})$ and natural
  transformations
  \begin{displaymath}
    0 \rTo \epsilon_{2}S \rTo \id \rTo \epsilon_{1}Q \rTo 0 \ .
  \end{displaymath}
  By \textsc{Quillen}'s additivity theorem \cite[Corollary~1,
  p.~106]{MR0338129} we have $(\epsilon_{2}S)_{*} +
  (\epsilon_{1}Q)_{*} = \id$, and this sum factors as
  \begin{multline*}
    K_{n} \big( \idem\lt(\cA_{1},\cA_{2}) \big) \rTo^{(Q_{*},S_{*})}
    K_{n} (\idem\cA_{1}) \oplus K_{n} (\idem \cA_{2}) \\
    \rTo^{\epsilon_{1*}+\epsilon_{2*}}  K_{n} \big(
    \idem\lt(\cA_{1},\cA_{2}) \big) \ .
  \end{multline*}
  On the other hand $(Q_{*},S_{*}) \circ
  (\epsilon_{1*}+\epsilon_{2*})$ is the identity map of the group
  $K_{n} (\idem\cA_{1}) \oplus K_{n} (\idem \cA_{2})$. Hence
  $(Q_{*},S_{*})$ is an isomorphism with inverse
  $\epsilon_{1*}+\epsilon_{2*}$.
\end{proof}

\section{Generalisations}
\label{sec:generalisations}

Let $\cA$ be an additive category as before, and let $\cA_{q}$, $1
\leq q \leq r$, be a finite collection of full additive
subcategories. We define the lower triangular category $\lt(\cA_{q};
\, 1 \leq q \leq r)$ to be the category which has objects
\begin{displaymath}
  F = \bigoplus_{q=1}^{r} F_{q} \ , \quad F_{j} \in \cA_{j}
\end{displaymath}
(the direct sum decomposition being part of the data), and has
morphisms the lower triangular matrices
\begin{displaymath}
  f = \mat{f_{11}\\f_{21}&f_{22}\\ \vdots & \vdots & \ddots \\ f_{r1}
    & f_{r2} & \cdots & f_{rr}} \colon \bigoplus_{q=1}^{r} F_{q}  = F \rTo G
  = \bigoplus_{q=1}^{r} G_{q}
\end{displaymath}
with $f_{ij} \in \cA(F_{j}, G_{i})$. The category $\lt(\cA_{q}; \, 1
\leq q \leq r)$ is an additive category, and we have $\lt(\cA_{1},
\cA_{2}) = \lt(\cA_{q}; \, 1 \leq q \leq 2)$.

\begin{proposition}
  \label{prop:generalisations}
  Let $\cA$ be an additive category as before, and let $\cA_{q}$, $1
  \leq q \leq r$, be a finite collection of full additive
  subcategories. There are exact (=additive) functors
  \begin{gather*}
    T_{k} \colon \idem \lt(\cA_{q}; \, 1
    \leq q \leq r) \rTo \idem \cA_{k} \ , \\
    \Big( \bigoplus_{q=1}^{r} P_{q},\, \smm{p_{11}\\p_{21}&p_{22}\\ \vdots & \vdots & \ddots \\ p_{r1}
    & p_{r2} & \cdots & p_{rr}} \Big) \mapsto (P_{k}, p_{kk}) \ , \\
    \smm{f_{11}\\f_{21}&f_{22}\\ \vdots & \vdots & \ddots \\ f_{r1}
    & f_{r2} & \cdots & f_{rr}} \mapsto f_{kk} \ .
  \end{gather*}
  These functors induce an isomorphism on $K$-groups
  \begin{displaymath}
    (T_{1*},\, T_{2*},\, \cdots,\, T_{r*}) \colon K_{n} \big( \idem
    \lt(\cA_{q}; \, 1 \leq q \leq r) \big) \rTo \bigoplus_{k=1}^{r}
    K_{n} \big( \idem \cA_{k} \big) \ ;
  \end{displaymath}
  the inverse isomorphisms are induced by the functor
  \begin{align*}
    \prod_{q=1}^{r} \idem \cA_{q} & \rTo \idem \lt(\cA_{q}; \, 1 \leq
    q \leq r) \ , \\
    \big( (P_{q}, p_{qq}) \big)_{q=1}^{r} & \mapsto 
    \Big( \bigoplus_{q=1}^{r} P_{q},\, \smm{p_{11}\\&p_{22}\\&& \ddots
      \\&&& p_{rr}} \Big) \ .
  \end{align*}
\end{proposition}

\begin{proof}
  In view of the obvious identification
  \begin{displaymath}
    \lt(\cA_{q}; \, 1 \leq q \leq r) = \lt \big( \lt(\cA_{q}; \, 1
    \leq q \leq r-1) ,\, \cA_{r}\big) \ ,
  \end{displaymath}
  and the analogous equality for idempotent completions, this follows
  from a straightforward induction on~$r$. For $r=2$, the Proposition
  has been verified in Lemmas~\ref{lemma:additive_functors},
  \ref{lem:ses_functors} and~\ref{lem:K_split}; for $r \leq 1$ the
  Proposition is trivial. (Alternatively, consider the~$T_{k}$ as
  successive quotients of an admissible filtration of the identity
  functor, and apply Corollary~2, p.~107 of~\cite{MR0338129}.)
\end{proof}

\section{The application template}
\label{sec:template}

For the reader's convenience we end this part with a basic ``template''
for applying the abstract machinery. The template has to be adjusted
to the actual situation under consideration, as seen in our
applications later in the paper.

Given an additive category~$\cA$, the task is to compute $K_{n} (\idem
\cA)$. We proceed following these steps:
\begin{enumerate}[label={\rm (AT\arabic*)}]
\item \label{at:filter} Filter the category $\cA$ by full
  subcategories $\cA[S]$, where $S$ varies over the directed poset of
  non-empty finite subsets of a partially ordered set~$G$ or, more
  generally, over any upwards directed sub-poset of the power-set
  of~$G$ with union~$G$; ordering is by inclusion. (The poset
  structure of~$G$ is irrelevant at this stage.)  This yields a
  corresponding filtration $(\idem \cA)[S] = \idem (\cA[S])$ of $\idem
  \cA$.
\item \label{at:identify} For $S = \{s\}$ a one-element set, identify
  $\cA[s] = \cA[S]$ with some other interesting additive category
  $\cA'[s]$. This yields automatically an identification of $\idem
  \cA[s]$ with $\idem \cA'[s]$, by Lemma~\ref{lem:equivalent}.
\item \label{at:lt} For $S = \{s_{1},\, s_{2},\, \cdots,\, s_{r}\}$
  with $r \geq 2$ identify $\cA[S]$ with the category $\lt \big(
  \cA[\{s_{q}\}]; \, 1 \leq q \leq r \big)$. (Here we will make use of
  the order relation of~$G$ which will influence the indexing of the
  elements~$s_{j}$.)
\item \label{at:identify_K} From
  Proposition~\ref{prop:generalisations}, and from
  steps~\ref{at:identify} and~\ref{at:lt}, we obtain isomorphisms
  $\bigoplus_{s \in S} K_{n} \big( \idem \cA'[s] \big) \rTo^{\iso}
  K_{n} ( \idem \cA[S])$ which are natural in~$S$ with respect to set
  inclusion.
\item \label{at:final} As $\idem \cA = \bigcup_{S} \idem \cA[S]$ by
  step~\ref{at:filter}, and as $K$-theory commutes with filtered
  unions \cite[p.~104]{MR0338129}, we obtain the isomorphism
  \begin{displaymath}
    K_{n} (\idem \cA) \iso \bigoplus_{s \in G} K_{n} \big(\idem
    \cA'[s] \big) \ .
  \end{displaymath}
\end{enumerate}

If $\cA'[s] = \cA'$ does not depend on~$s$ we obtain an isomorphism
\begin{equation}
  \label{eq:template}
  K_{n} (\idem \cA) \iso \bigoplus_{s \in G} K_{n} (\idem \cA')
  \iso \bZ[G] \tensor_{\bZ} K_{n} (\idem \cA') \ ,
\end{equation}
where $\bZ[G]$ denotes the free \textsc{abel}ian group with basis~$G$.

\part{Systematic algebra}
\label{part:systematic_algebra}

\section{Systematic rings and modules}

Given a subset $B$ of a ring~$R$ and a subset $A$ of a right $R$-module,
we let $AB$ denote the set of finite sums of products $ab$ with $a \in
A$ and~$b \in B$. --- Let $G$ be a group, multiplicatively written. A
{\it (unital) $G$-systematic ring} is a unital ring $R$ together with
a family $(R_{g})_{g \in G}$ of additive subgroups of $(R,+)$ such
that
\begin{enumerate}[label={\rm (SR\arabic*)}]
\item $R = \sum_{g \in G} R_{g}$ (that is, the subgroups $R_{g}$
  generate~$R$ as an \textsc{abel}ian group),
\item $R_{g}R_{h} \subseteq R_{gh}$ for all $g,h \in G$,
\item $1 \in R_{1}$.
\end{enumerate}
The first two conditions define what is called a $G$-system by
\textsc{Grzeszczuk} \cite{MR806068}. The last condition $1 \in R_{1}$
is redundant for finite~$G$, see \cite[Theorem~1]{MR806068}. --- The
unital $G$-systematic rings are precisely the homomorphic images of
unital $G$-graded rings. Indeed, if $\pi \colon R' \rTo R$ is a
surjective ring homomorphism with $R'$ a $G$-graded ring, then setting
$R_{g} = \pi(R'_{g})$ makes $R$ into a $G$-systematic
ring. Conversely, if $R$ is $G$-systematic define $R'_{g} = \{g\}
\times R_{g}$ and $R' = \bigoplus_{g \in G} R'_{g}$; this is a
$G$-graded ring with multiplication determined by $(g,a) \cdot (h,b) =
(gh, ab)$, and the obvious map $\pi \colon (g,a) \mapsto a$ is a
surjective ring homomorphism. Note that $\ker \pi$ need not be a
graded ideal.

A filtered ring~$R$ equipped with an increasing or decreasing
filtration $(F^{k}R)_{k \in \bZ}$ can be considered as a
$\bZ$-systematic ring by setting $R_{k} = F^{k} R$, provided that $R =
\bigcup_{k} F^{k}R$ and $1 \in F^{0}R$.

\medbreak

Given a $G$-systematic ring~$R$, a {\it $G$-systematic $R$-module} is
a unital right $R$\nbd-module $M$ together with a family $(M_{g})_{g
  \in G}$ of additive subgroups of $(M,+)$ such that
\begin{enumerate}[label={\rm (SM\arabic*)}]
\item $M = \sum_{g \in G} M_{g}$ (that is, the subgroups $M_{g}$
  generate~$M$ as an \textsc{abel}ian group),
\item $M_{g}R_{h} \subseteq M_{gh}$ for all $g,h \in G$.
\end{enumerate}
A homomorphism $f \colon M \rTo N$ of $G$-systematic modules is an
$R$-linear map such that $f (M_{g}) \subseteq N_{g}$ for all $g \in
G$. Direct sums are given by the prescription
\begin{displaymath}
  \Big( \sum_{g \in G} M_{g} \Big) \oplus \Big( \sum_{g \in G} N_{g}
  \Big) = \sum_{g \in G} \big( M_{g} \oplus N_{g} \big) \ .
\end{displaymath}
This defines the additive category $\syst G R$ of $G$-systematic
$R$-modules. --- Every $R$-module $M$ can be considered as a
$G$-systematic $R$-module when equipped with the trivial systematic
structure $M_{g} = M$. This defines a functor from the category of
$R$-modules to the category of systematic $R$\nbd-modules which is
right adjoint to the functor which forgets the systematic structure.

For $M$ a $G$-systematic module and $a$~an element of~$G$ we use the
symbol $\shift a M$ to denote the $a$-shift of~$M$; this is the
$G$-systematic module which is $M$ as an $R$-module, with systematic
structure given by $\shift a M_{g} = M_{a\inv g}$. Clearly $\shift b
{\big( \shift a M \big)} = \shift {ba} M$ so that shifting defines a
left $G$\nbd-action on the category $\syst G R$.

\section{Systematically free and projective modules}

Let $R$ be a $G$-systematic ring. We are interested in the category
$\cP_G$ of finitely generated systematically projective $R$-modules,
which are direct summands of direct sums of modules of the
form~$\shift g R$. In other words, $\cP_G$ is the idempotent
completion of the additive category~$\cF_G$ of finitely generated
systematically free modules, which has objects all finite direct sums
with summands of the form~$\shift g R$. We will in fact work with {\it
  systematically free based modules\/} throughout, that is, free
modules equipped with a choice of preferred basis elements. (Morphisms
are not required to respect basis elements.)  The preferred generator
of~$\shift g R$ is the unit element $1 \in \shift g R_{g}$.

Given a set $S \subseteq G$ we let $\cF_G[S]$ denote the full additive
subcategory of $\syst G R$ with objects the $G$-systematically free
based modules of the form
\begin{displaymath}
  \bigoplus_{s \in S} (\shift s R)^{m_{s}}
\end{displaymath}
for integers $m_{s} \geq 0$, of which only finitely many are allowed
to be non-zero.  This is the additive category of {\it finitely
  generated systematically free based modules with generators having
  degrees in~$S$}. We denote the idempotent completion $\idem\cF_G[S]$
by $\cP_G[S]$, and call $\cP_G[S]$ the category of {\it finitely
  generated systematically projective modules with generators having
  degrees in~$S$}.

Given sets $S \subseteq T \subset G$ we have inclusions of full
subcategories $\cF_G[S] \subseteq \cF_G[T]$ and $\cP_G[S] \subseteq
\cP_G[T]$, resulting in systems of additive categories indexed by the
power set of~$G$ (ordered by inclusion). In view of
Lemma~\ref{lem:filtration} we observe
\begin{equation}
  \label{eq:syst_filtered}
  \cF_G = \bigcup_{S} \cF_G[S] \qquad \text{and} \qquad \cP_G =
  \idem \cF_G = \bigcup_{S} \cP_G[S]
\end{equation}
whenever we let $S$ vary over the full power set of~$G$, or some
upwards directed sub-poset covering all of~$G$.

\section{Strongly systematic rings and modules}

\begin{definition}
  We call a $G$-systematic ring~$K$ (\resp, a $G$\nbd-systematic
  $K$\nbd-module~$M$) {\it strongly systematic\/} if for all $g_{1},
  g_{2} \in G$ we have an equality $K_{g_{1}} K_{g_{2}} =
  K_{g_{1}g_{2}}$ (\resp, $M_{g_{1}} K_{g_{2}} = M_{g_{1}g_{2}}$).
\end{definition}

Strongly systematic rings are exactly the homomorphic images of
strongly graded rings, and coincide with ``\textsc{Clifford} systems''
in the sense of \textsc{Dade} \cite{MR0262384}, and with ``almost
strongly graded rings'' satisfying $1 \in K_{1}$ in the terminology of
\textsc{N{\u{a}}st{\u{a}}sescu} and \textsc{van Oystaeyen}
\cite[\S{}I.8]{MR676974}. --- Clearly a strongly $G$-systematic
ring~$K$ satisfies $1 \in K_{g\inv} K_{g}$ for all $g \in
G$. Conversely, suppose that $K$ is a $G$-systematic ring with $1 \in
K_{h\inv} K_{h}$ for all $h \in G$. Then
\begin{displaymath}
  K_{g}K_{h} \subseteq K_{gh} \subseteq K_{gh} (K_{h\inv} K_{h})
  \subseteq K_{g} K_{h}
\end{displaymath}
so that $K_{g} K_{h} = K_{gh}$ for all $g,h \in G$: The ring $K$~is
strongly systematic.

\begin{lemma}
  \label{lem:systematic_Dade}
  Let $Q$ be a group, let $K = \sum_{g \in Q} K_{g}$ be a
  $Q$-systematic ring, and let $M = \sum_{g \in Q} M_{g}$ be a
  $Q$-systematic $K$-module. If $K$ is strongly systematic then so
  is~$M$.
\end{lemma}

\begin{proof}
  For strongly systematic $K$, since $1 \in K_{1}$ we have
  \begin{displaymath}
    M_{g}K_{h} \subseteq M_{gh} = M_{gh} K_{1} = M_{gh} K_{h\inv}
    K_{h} \subseteq M_{g} K_{h}
  \end{displaymath}
  for all $g,h \in Q$. That is, we have $M_{g} K_{h} = M_{gh}$ as
  required.
\end{proof}

\begin{lemma}[{cf.~\textsc{N{\u{a}}st{\u{a}}sescu} and \textsc{van
      Oystaeyen} \cite[I.8.2]{MR676974}}]
  \label{lem:K_a-fgp}
  Let $K$ be a $Q$-systematic ring, and let $a \in Q$ be such that $1
  \in K_{a}K_{a\inv}$. Then the right $K_{1}$-module $K_{a}$ is
  finitely generated projective.
\end{lemma}

\begin{proof}
  By hypothesis there is a finite sum decomposition $1 = \sum_{j}
  \alpha_{j} \beta_{j}$ with $\alpha_{j} \in K_{a}$ and $\beta_{j}
  \in K_{a\inv}$. Define $K_{1}$-linear maps
  \begin{displaymath}
    \rho_{j} \colon K_{a} \rTo K_{1} \ , \quad r \mapsto \beta_{j} r \ ,
  \end{displaymath}
  and observe that $\sum_{j}
  \alpha_{j} \rho_{j} = \id_{K_{a}}$ as
  \begin{displaymath}
    \sum_{j} \alpha_{j} \rho_{j} (r)= \sum_{j} \alpha_{j} \beta_{j} r
    = 1 \cdot r = r
  \end{displaymath}
  so that the collection of the~$\alpha_{j}$ and~$\rho_{j}$ form a
  dual basis of~$K_{a}$. Consequently $K_{a}$ is finitely generated
  projective as a right $K_{1}$-module, see \textsc{Bourbaki}
  \cite[\S{}II.2.6, Proposition~12]{MR1727844}.
\end{proof}

Let $K$ be a $Q$-systematic ring, and let $L$ be a right
$K_{1}$-module. We consider the tensor product $L \tensor_{K_{1}} K$
as a $Q$-systematic $K$-module with $(L \tensor_{K_{1}} K)_{q}$ the
subgroup generated by the primitive tensors $\ell \tensor k$ with
$\ell \in L$ and $k \in K_{q}$. In other words, $(L \tensor_{K_{1}}
K)_{q}$ is the image of map $\omega \colon L \tensor_{K_{1}} K_{q}
\rTo L \tensor_{K_{1}} K$. The map~$\omega$ is injective in case
$L$~is a projective (or, more generally, flat) $K_{1}$-module, in
which case we will tacitly identify $L \tensor_{K_{1}} K_{q}$ with $(L
\tensor_{K_{1}} K)_{q}$.

\begin{lemma}
  \label{lem:shift_by_tensor}
  Let $K$ be a $Q$-systematic ring, and let $a \in Q$ be such that $1
  \in K_{a\inv}K_{a}$. The natural map
  \begin{displaymath}
    \nu_{\shift a K} \colon K_{a\inv} \tensor_{K_{1}} K \rTo \shift a K \ , s
    \tensor r  \mapsto sr
  \end{displaymath}
  is an isomorphism of $Q$-systematic $K$-modules.
\end{lemma}

\begin{proof}
  By hypothesis there is a finite sum decomposition $1 = \sum_{j}
  \alpha_{j} \beta_{j}$ with $\alpha_{j} \in K_{a\inv}$ and $\beta_{j}
  \in K_{a}$. Define
  \begin{displaymath}
    \tau \colon \shift a K \rTo K_{a\inv} \tensor_{K_{1}} K \ , x
    \mapsto \sum_{j} \alpha_{j} \tensor (\beta_{j} x) \ .
  \end{displaymath}
  This map is clearly $K$-linear, and it is systematic as for $x \in
  \shift a K _{g} = K_{a\inv g}$ we have $\beta_{j} x \in
  K_{a}K_{a\inv g} \subseteq K_{g}$. Now $\nu_{\shift a K} \circ \tau (x) =
  \sum_{j} \alpha_{j} \beta_{j} x = x$, and
  \begin{displaymath}
    \tau \circ \nu_{\shift a K} (s \tensor r) = \tau (sr) = \sum_{j} \alpha_{j}
    \tensor (\beta_{j} sr) \ .
  \end{displaymath}
  But $s \in K_{a\inv}$ so that $\beta_{j} s \in K_{1}$; consequently,
  $\alpha_{j} \tensor (\beta_{j} sr) = (\alpha_{j} \beta_{j} s)
  \tensor r$, with $\alpha_{j} \beta_{j} \in K_{1}$, so that
  \begin{displaymath}
    \tau \circ \nu_{\shift a K} (s \tensor r) = \sum_{j} \alpha_{j} \beta_{j}
    \cdot s \tensor r = s \tensor r \ .
  \end{displaymath}
  We have shown that both compositions $\tau \circ \nu_{\shift a K}$ and
  $\nu_{\shift a K} \circ \tau$ are identity maps, as required.
\end{proof}

\begin{proposition}
  \label{prop:strong_systematic_equivalence}
  Let $Q$ be a group, and let $K$ be a strongly $Q$-systematic ring.
  Then the category $\cP_{Q}$ of finitely generated $Q$-systematically
  projective $K$-modules is equivalent to the category of finitely
  generated projective $K_{1}$-modules via the functor~$\rho$ that
  sends a module $M = \sum_{q \in Q} M_{q}$ to its
  component~$M_{1}$. The inverse equivalence is given by the functor
  $\tau$ sending $L$ to the $Q$-systematic module $L \tensor_{K_{1}} K
  = \sum_{q \in Q} L \tensor_{K_{1}} K_{q}$.
\end{proposition}

\begin{proof}
  As $K$ is strongly systematic, each of its components $K_{a}$ is a
  finitely generated projective $K_{1}$-module by
  Lemma~\ref{lem:K_a-fgp}. It follows that $\rho \colon M \mapsto
  M_{1}$ maps finitely generated systematically projective $K$-modules
  to finitely generated projective $K_{1}$-modules.

  If $L$ is a finitely generated free $K_{1}$-module then $L
  \tensor_{K_{1}} K$ is (isomorphic to) a finite direct sum of copies
  of~$K$, that is, $L \tensor_{K_{1}} K$ is a finitely generated
  systematically free module. It follows that $\tau \colon L \mapsto
  L \tensor_{K_{1}} K$ maps finitely generated projective
  $K_{1}$-modules to finitely generated systematically projective
  $K$-modules.

  For a projective $K_{1}$-module $L$ we have $(L \tensor_{K_{1}}
  K)_{1} = L \tensor_{K_{1}} K_{1} \iso L$ (natural in~$L$) so that
  $\rho \circ \tau \iso \id$.

  Define a natural transformation $\nu \colon \tau \circ \rho \rTo^\natdot
  \id$ by
  \begin{displaymath}
    \nu_{M} \colon M_{1} \tensor_{K_{1}} K \rTo M \ , \quad (m,k)
    \mapsto mk \ .
  \end{displaymath}
  This is a map of systematic modules: the image of $(M_{1}
  \tensor_{K_{1}} K)_{q} = M_{1} \tensor_{K_{1}} K_{q}$ is contained
  in $M_{q}$. As $K$ is strongly systematic we have $1 \in K_{1} =
  K_{a\inv} K_{a}$ for every $a \in Q$. Hence for $M = \shift a K$ a
  systematically free module on one generator the map $\nu_{M} =
  \nu_{\shift a K}$ is an isomorphism by
  Lemma~\ref{lem:shift_by_tensor}. By allowing direct sums we see that
  $\nu_{M}$ is an isomorphism for (finitely generated) systematically
  free modules; further allowing retracts of free modules yields the
  statement of the proposition.
\end{proof}

\begin{remark}
  \begin{enumerate}
  \item The proof shows that we also obtain an equivalence between the
    category of $Q$-systematically projective $K$-modules and the
    category of projective $K_{1}$-modules (without any finite
    generation hypothesis). However, the theorem does not extend to
    all modules, nor to all finitely generated modules. For example,
    consider $K = \bZ[1/2]$ as a $\bZ$-systematic ring with systematic
    structure defined by $K_{k} = \{ 2^{k} \cdot x \,|\, x \in \bZ\}$
    so that $K_{0} = \bZ$. Then $L = \bZ/2$ is a non-trivial
    $K_{0}$-module, but $\tau(L) = L \tensor_{K_{0}} K =(\bZ/2)
    \tensor_{\bZ} \bZ[1/2] = 0$ so that $\rho \circ \tau(L) = (L
    \tensor_{K_{0}} K)_{0} \not\iso L$.
  \item The functor $\rho$ is additive and preserves injections, but
    not surjections. For example, with $K = \bZ[1/2]$ as above,
    multiplication by~$2$ is a surjective (in fact, bijective)
    self-map of~$K$; application of~$\rho$ results in the self-map
    of~$\bZ$ given by multiplication by~$2$, which is not onto.
  \end{enumerate}
\end{remark}

\part{Algebraic $K$-theory}
\label{part:systematic_K}

\section{Systematic $K$-theory}

\subsection*{Systematic $K$-theory}

The {\it $G$-systematic $K$-theory of~$R$} is the algebraic $K$-theory
of the additive category $\cP_{G}$ of finitely generated
$G$-systematically projective $R$-modules with respect to the split
exact structure. We write $K_{n}^{\Sy G} (R) = K_{n} (\cP_{G})$. As
$G$ acts on the left of the category $\cP_{G}$ by shifting, the groups
$K_{n}^{\Sy G} (R)$ acquire a left $\bZ[G]$-module structure described
by
\begin{displaymath}
  (g, P) \mapsto \shift g P \qquad \text{for \(g \in G\)} \ .
\end{displaymath}

\section{Systematic $K$-theory of rings with positive support}

Suppose now we are given an extension of multiplicative groups
\begin{equation}
\label{eq:ext}
  1 \rTo N \rTo G \rTo H \rTo 1 \ .
\end{equation}
Our goal is to analyse the $G$-systematic $K$-theory of a
$G$-systematic ring~$R$ in terms of the $N$-systematic or
$H$-systematic $K$-theory of certain subrings of~$R$, under the
assumption that $R$ has ``positive support'' in a suitable sense. This
requires, among other things, to have a partial order on either $N$
or~$H$.

\subsection*{Ordering the subgroup}

Suppose that the extension~\eqref{eq:ext} is split so that $G = N
\rtimes H$ is a semi-direct product with respect to some group
homomorphism $\theta \colon H \rTo \mathrm{Aut} (N)$. We simply write
$hn$ for $\theta(h)(n)$, so that the group multiplication takes the
form $(n,h)(n',h') = (n \cdot hn', hh')$, and inverses are computed as
$(n,h)\inv = (h\inv n\inv, h\inv)$. Given a $G$-systematic module $M$
we define
\begin{displaymath}
  M^{H} = \sum_{h \in H} M_{(1,h)} \ .
\end{displaymath}
Applied to $M=R$ this provides us with the $H$-systematic ring
$R^{H}\!$, and in general $M^{H}$ is an $H$-systematic $R^{H}$-module
by restriction of scalars.

\begin{lemma}
  \label{lem:systematic_equiv}
  Let $s \in N$ be a fixed element. The category $\cF_{H}$ of finitely
  generated $H$\nbd-systematically free $R^{H}$-modules is equivalent
  to the category $\cF_{N \rtimes H}[s \rtimes H]$. The equivalence
  takes the module $\shift h R^{H}$ to the module $\shift {s,h} R$.
\end{lemma}

\begin{proof}
  Finitely generated free based modules are entirely determined by
  their generators. Hence there is a bijection of objects determined
  by the assignment
  \begin{displaymath}
    \shift h R^{H} \mapsto \shift {s,h} R \ .
  \end{displaymath}
  Morphisms $\shift {s,h_{1}} R \rTo \shift {s,h_{2}} R$ in $\cF_{N
    \rtimes H}[s \rtimes H]$ are in bijective correspondence with
  elements of
  \begin{displaymath}
    \shift {s,h_{2}} R _{(s,h_{1})} = R_{(s,h_{2})\inv (s,h_{1})} =
    R_{(1, h_{2}\inv h_{1})} \ ,
  \end{displaymath}
  as morphisms of free modules on one generator are determined by the
  image of the generator. On the other hand, morphisms $\shift {h_{1}}
  R^{H} \rTo \shift {h_{2}} R^{H}$ in $\cF_{H}$ are in bijective
  correspondence with elements of
  \begin{displaymath}
    \shift {h_{2}} R^{H} _{h_{1}} = R^{H}_{h_{2}\inv h_{1}} = R_{(1,
      h_{2}\inv h_{1})} \ ,
  \end{displaymath}
  which is the same set. Composition of morphisms corresponds to
  multiplication of elements in both categories. This shows that the
  set of morphisms $\shift {s,h_{1}} R \rTo \shift {s,h_{2}} R$ and
  the set of morphisms $\shift {h_{1}} R^{H} \rTo \shift {h_{2}}
  R^{H}$ are the same. The claim follows by considering finite direct
  sums of modules (which amounts to considering matrices instead of
  single ring elements).
\end{proof}

\begin{theorem}
  \label{thm:generalised_Q}
  Let $G = N \rtimes H$ be a semi-direct product as before.  Suppose
  that $N$ is equipped with a translation-invariant partial order (so
  that $n \geq n'$ implies $anb \geq an'b$ for all $a,b,n,n' \in N$)
  which is also invariant under the action of~$H$ (so that $n \geq n'$
  implies $hn \geq hn'$ for all $n,n' \in N$ and $h \in H$). Let $R =
  \sum_{(n,h) \in N \rtimes H} R_{(n,h)}$ be an $(N \rtimes
  H)$-systematic ring with support in $N^{+} \rtimes H$, where $N^{+}
  = \{n \in N \,|\, n \geq 1\}$ is the positive cone of~$N$ as usual;
  that is, suppose that $R_{(n,h)} = \{0\}$ whenever $n \notin
  N^{+}$. There are canonical $\bZ[N \rtimes H]$-module isomorphisms
  \[\bZ[N \rtimes H] \tensor_{\bZ[H]} K^{\Sy H}_{n} (R^{H})
  \iso K^{\Sy {(N \rtimes H)}}_{n}(R) \ ,\]
  described by the formula $(s,h) \tensor \shift g R^{H} \mapsto \shift
    {s,hg} R$.
\end{theorem}

\begin{proof}
  {\bf Step~\ref{at:filter}.} The category $\cF_{G}$ is the filtered
  union of the categories $\cF_{G}[S \rtimes H]$, where $S$ ranges
  over the directed poset $\pow N$ of non-empty finite subsets of~$N$
  (ordered by inclusion). Consequently, $\cP_{G}$ is the filtered
  union of the categories~$\cP_{G}[S \rtimes H]$ by
  Lemma~\ref{lem:filtration}.

  {\bf Step~\ref{at:identify}.} It has been observed in
  Lemma~\ref{lem:systematic_equiv} that $\cF_{G}[s \rtimes H]$ is
  equivalent to the category~$\cF_{H}$. By Lemma~\ref{lem:equivalent},
  the categories $\cP_{G}[s \rtimes H]$ and $\cP_{H}$ are consequently
  equivalent.

  {\bf Steps~\ref{at:lt} and~\ref{at:identify_K}.} Let $S$ be a finite
  non-empty subset of~$N$. We write $S$ in the form $S = \{s_{1},\,
  s_{2},\, \cdots,\, s_{r}\}$ where $s_{i} > s_{j}$ implies $i<j$
  (that is, we choose a linear extension of the poset~$S$). We claim
  that the categories $\cF_{G}[S \rtimes H]$ and $\lt\big(
  \cF_{G}[s_{q} \rtimes H]; 1 \leq q \leq r\big)$ are
  equivalent. Every object~$P$ of~$\cF_{G}[S \rtimes H]$ is, by
  definition, a direct sum of the form $P = P_{1} \oplus P_{2} \oplus
  \ldots \oplus P_{r}$, where
  \begin{displaymath}
    P_{j} = \bigoplus_{k=1}^{n_{j}} \shift {s_{j}, h_{jk}} R
  \end{displaymath}
  with certain integers $n_{j} \geq 0$ and element $h_{jk} \in H$.
  That is, $P$ is a collection of objects $P_{j}$ of~$\cF_{G}[s_{j}
  \rtimes H]$. A morphism~$f$ from $P = P_{1} \oplus P_{2} \oplus
  \ldots \oplus P_{r}$ to $Q = Q_{1} \oplus Q_{2} \oplus \ldots \oplus
  Q_{r}$ is a matrix $(f_{ij})$ with $f_{ij} \colon P_{j} \rTo Q_{i}$
  a $G$-systematic map. To prove the desired equivalence of categories
  it is enough to verify that this matrix is necessarily a lower
  triangular matrix.

  We thus want to show that for $i < j$ we must have $f_{ij} =
  0$. Indeed, it is enough to show that every systematic homomorphism
  $\shift {s_{j}, h_{jk}} R \rTo \shift {s_{i}, h_{i\ell}} R$ is
  necessarily trivial. For the latter is determined by the image of
  the generator $1 \in \shift {s_{j}, h_{jk}} R _{(s_{j}, h_{jk})}$ in
  the set
  \begin{multline*}
    \shift {s_{i}, h_{i\ell}} R_{(s_{j}, h_{jk})} = R_{(s_{i}, h_{i\ell})\inv
      \cdot (s_{j}, h_{jk})} = R_{(h_{i\ell}\inv s_{i}\inv,
      h_{i\ell}\inv) \cdot (s_{j}, h_{jk})} \\ = R_{(h_{i\ell}\inv
      s_{i}\inv \cdot h_{i\ell}\inv s_{j}, h_{i\ell}\inv h_{jk})} =
    R_{(h_{i\ell}\inv (s_{i}\inv s_{j}), h_{i\ell}\inv h_{jk})} \ .
  \end{multline*}
  As $j > i$ we cannot have $s_{j} \geq s_{i}$, by choice of linear
  extension, and as the order is invariant under the action of~$H$ by
  hypothesis, $h_{i\ell}\inv s_{j} \not\geq h_{i\ell}\inv s_{i}$ and
  thus
  \begin{displaymath}
    h_{i\ell}\inv (s_{i}\inv s_{j}) =
    h_{i\ell}\inv s_{i} \inv \cdot h_{i\ell}\inv s_{j} =
    (h_{i\ell}\inv s_{i})\inv \cdot h_{i\ell}\inv s_{j} \not\geq 1 \ .
  \end{displaymath}
  This means that $h_{i\ell}\inv (s_{i}\inv s_{j}) \notin N^{+}$
  whence $\shift {s_{i}, h_{i\ell}} R _{(s_{j}, h_{jk})}$ must be the
  trivial group, by our hypothesis on the support of~$R$. In other
  words, the generator of the source must map to~$0$, and so the
  homomorphism under investigation is trivial.

  By Proposition~\ref{prop:generalisations} the inclusion functors
  $\cP_{G}[s_{j} \rtimes H] \rTo \cP_{G}[S \rtimes H]$ induce
  isomorphisms $\bigoplus_{s \in S} K_{n}(\cP_{G}[s \rtimes H])
  \rTo^{\iso} K_{n}(\cP_{G}[S \rtimes H])$.  From the previous step we
  conclude that we have isomorphisms
  \begin{equation}
    \label{eq:iso}
    \bZ[S] \tensor_{\bZ} K^{\Sy H}_{n}(R^{H}) \iso \bigoplus_{s \in S}
    K^{\Sy H}_{n}(R^{H}) \rTo[l>=3em]^{\iso} K_{n}(\cP_{G}[S
    \rtimes H])
  \end{equation}
  (where \(\bZ[S]\) denotes the free \textsc{abel}ian group on~\(S\))
  which are determined by the assignments $s \tensor \shift h R^{H}
  \mapsto \shift {s,h} R$. These isomorphisms are natural with
  respect to set inclusions $S \subseteq T$.

  {\bf Step~\ref{at:final}.} Upon passing
  to the colimit over $S \in \pow N$ in~\eqref{eq:iso} we obtain an
  isomorphism
  \begin{displaymath}
    \omega \colon \bZ[N] \tensor_{\bZ} K^{\Sy H}_{n}(R^{H}) 
    \rTo^{\iso} K^{\Sy {(N \rtimes H)}}_{n}(R) \ , \quad s \tensor
    \shift h R^{H} \mapsto \shift {s,h} R
  \end{displaymath}
  of \textsc{abel}ian groups.

  We have maps of left $\bZ[N]$-modules
  \begin{displaymath}
    \bZ[N] \tensor_{\bZ} K^{\Sy H}_{n}(R^{H}) \pile{\rTo^{\alpha} \\
    \lTo_{\beta}} \bZ[N \rtimes H] \tensor_{\bZ[H]} K^{\Sy H}_{n}(R^{H}) 
  \end{displaymath}
  described by the formul\ae{}
  \begin{displaymath}
    \alpha (s \tensor P) = (s,1) \tensor P \qquad \text{and} \qquad
    \beta \big( (s,h) \tensor P \big) = s \tensor \shift h P \ ;
  \end{displaymath}
  that is, $\alpha$ is induced by the inclusion of rings $\bZ[N] \rTo
  \bZ[N \rtimes H]$. To show that $\beta$ is well defined, it suffices
  to note that the assignment $\big( (s,h),\, P \big) \mapsto s
  \tensor \shift h P$ is $\bZ[H]$-balanced: Indeed, we have
  \begin{gather*}
    \big( (s,h) \cdot (1,h'),\, P \big) = \big( (s,hh'),\, P \big)
    \mapsto s \tensor \shift {hh'} P \\
    \noalign{\noindent and}%
    \big( (s,h),\, \shift {h'} P \big) \mapsto s \tensor \shift h
    {\shift {h'} P} = s \tensor \shift{hh'} P \ .
  \end{gather*}
  As $\shift 1 P = P$ we have $\beta \circ \alpha = \id$. But
  \begin{multline*}
    \alpha \circ \beta \big( (s,h) \tensor P \big) = \alpha (s \tensor \shift
    h P) = (s,1) \tensor \shift h P \\ = (s,1) \cdot (1,h) \tensor P =
    (s,h) \tensor P \qquad
  \end{multline*}
  so that $\beta$ is in fact an isomorphism.

  The composite map
  \begin{displaymath}
    \omega \circ \beta \colon \bZ[N \rtimes H] \tensor_{\bZ[H]} K^{\Sy
      H}_{n}(R^{H}) \rTo K^{\Sy {(N \rtimes H)}}_{n}(R)
  \end{displaymath}
  is described by the formula
  \begin{displaymath}
    \omega \circ \beta \colon (s,h) \tensor \shift g R^{H} \mapsto \shift
    {s,hg} R \ ,
  \end{displaymath}
  as can be seen readily by tracing the definitions. It is bijective
  by what we have proved before; it is in fact an isomorphism of left
  $\bZ[N \rtimes H]$-modules as, on the one hand,
  \begin{displaymath}
    \omega \circ \beta \big((s',h')\cdot (s,h) \tensor \shift g R^{H} \big)
    = \omega \circ \beta \big( (s' \cdot h's, h'h) \tensor \shift g R^{H}
    \big) = \shift {s' \cdot h's, h'hg} R
  \end{displaymath}
  while, on the other hand,
  \begin{displaymath}
  \shift {s',h'} {\omega \circ \beta \big((s,h) \tensor \shift g R^{H} \big)}
    = \shift {s',h'} {\shift {s,hg} R} = \shift {s' \cdot h's, h'hg} R
    \ . \tag*{\qedsymbol}
  \end{displaymath}
  \noqed%
\end{proof}

The theorem can be applied in rather more common situations. For
example, a $G$-graded ring $R = \bigoplus_{g \in G} R_{g}$ may be
considered as a $G$-systematic ring. The category $\cF_{G}$ is then
equivalent to the category of finitely generated $G$-graded free
$R$-modules so that Theorem~\ref{thm:generalised_Q} gives a
description of the graded $K$-theory of~$R$. This unifies various
calculations of graded $K$-theory in the literature, for example those
by \textsc{Quillen} \cite[p.~107, Proposition]{MR0338129} (the case $N
= \bZ$ and $H = \{1\}$), \textsc{Hazrat} and \textsc{H\"uttemann}
\cite[Theorem~1]{MR3095326} (the case of a direct product $\bZ \times
H$) and \textsc{H\"uttemann} \cite{MR3169434} (the case $N = \bZ^{n}$
and $H = \{1\}$, and $R$ having support in a polyhedral pointed cone).
--- It should be pointed out explicitly that the category of {\it
  all\/} systematic modules over a graded ring can contain non-graded
modules, but that a {\it systematically projective\/} module over a
graded ring is automatically graded.

\medbreak

The theorem also applies to a positively filtered ring~$R$, that is, a
ring equipped with an ascending chain of additive subgroups $F^{k}R$,
$k \in \bZ$, such that $F^{-1}R = \{0\}$, $1 \in F^{0}R$, $F^{k}R
\cdot F^{\ell} R \subseteq F^{k + \ell} R$, and $\bigcup_{k} F^{k}R =
R$. Setting $R_{k} = F^{k}R$ we obtain a $\bZ$-systematic ring; the
category $\cF_{\bZ}$ is now equivalent to the category of finitely
generated filt-free based $R$-modules in the sense of
\textsc{N{\u{a}}st{\u{a}}sescu} and \textsc{van Oystaeyen} \cite[\S
D~IV]{MR676974}. Application of Theorem~\ref{thm:generalised_Q} to the
trivial group extension $N = G = \bZ$ and $H = \{1\}$ then says that
the filtered $K$-theory of~$R$ is isomorphic to the tensor product of
$\bZ[x,x\inv] = \bZ[\bZ]$ with the (usual) $K$-theory of $F^{0}R$.

A positively filtered ring as above determines an associated
$\bZ$-graded ring with support in~$\bN$. Applying
Theorem~\ref{thm:generalised_Q} to both the filtered ring and the
associated graded ring, and noting that the ``degree 0''-pieces are
the same, we immediately get the following:

\begin{corollary}
  Let $R$ be a positively filtered ring, and let $B$ denote the
  associated $\bZ$-graded ring $\bigoplus_{k\geq0}
  F^{k}R/F^{k-1}R$. The filtered $K$-theory of~$R$ and the
  $\bZ$-graded $K$-theory of~$B$ are both isomorphic to $\bZ[x,x\inv]
  \tensor_{\bZ} K_{n} (F^{0}R)$ as left $\bZ[x,x\inv]$-modules; the
  action of the indeterminate~$x$ corresponds to shifting of filtered
  and graded modules, respectively.\qed
\end{corollary}

\subsection*{Ordering the quotient}

Given a (possibly non-split) group extension of multiplicative group
\begin{displaymath}
  1 \rTo N \rTo^{\subseteq} G \rTo^{\pi} H \rTo 1
\end{displaymath}
and a translation-invariant partial order ``$\geq$'' on~$H$, we write
$H^{+}$ for the positive cone $\{h \in H \,|\, h \geq 1\}$ of~$H$ and
define $G^{+} = \pi\inv (H^{+})$. Let $R = \sum_{g \in G} R_{g}$ be a
$G$-systematic ring, and let $R_{N} = \sum_{n \in N} R_{n}$ be the
$N$-systematic subring determined by~$N$.

\begin{theorem}
  \label{thm:syst_K}
  Suppose the $G$-systematic ring~$R$ has support in~$G^{+}$ in the
  sense that $R_{g} = \{0\}$ whenever $g \notin G^{+}$. Any choice of
  set-theoretic section $\sigma \colon H \rTo G$ of~$\pi$ determines
  isomorphisms of \textsc{abel}ian groups
  \begin{equation}
    \label{eq:1}
    \Psi \colon \bigoplus_{H} K^{\Sy N}_{q} (R_{N}) \rTo^{\iso} K^{\Sy
      G}_{q} (R) \ ;
  \end{equation}
  the restriction of this isomorphism to the $h$-summand is induced by
  the functor determined by $\shift n R_{N} \mapsto \shift {\sigma(h)
    \cdot n} R$.
\end{theorem}

\begin{proof}
  The proof is similar to that of Theorem~\ref{thm:generalised_Q}; we
  give a short account of the relevant arguments. --- Let $g,g' \in G$
  be such that $g\inv g' \notin G^{+}$, \ie, such that $\pi(g)
  \not\leq \pi(g')$. Then there are no non-trivial homomorphisms $\eta
  \colon \shift {g'} R \rTo \shift g R$. Indeed, $\eta$ is determined
  by the image of the generator $1 \in \shift {g'} R _{g'}$, that is,
  by the element $\eta(1) \in \shift g R _{g'} = R_{g\inv g'}$; but
  $R_{g\inv g'} = \{0\}$ by our condition on the support of~$R$.

  Now let $S = \{s_{1},\, s_{2},\, \cdots,\, s_{r}\} \subseteq H$ be a
  finite subset; we may assume, by renumbering if necessary, that
  $s_{i} > s_{j}$ implies $i < j$. It follows from the previous
  paragraph that we can identify $\cF_{G}[\pi\inv S]$ with the lower
  triangular category $\lt\big(\cF_{G}[\pi\inv(s_{q})]; 1 \leq q \leq
  r\big)$; from Proposition~\ref{prop:generalisations} we infer that
  the inclusion functors induce isomorphisms
  \begin{equation}
    \label{eq:iso_for_S}
    \bigoplus_{s \in S} K_{q} \big( \cP_{G}[\pi\inv(s)] \big) \rTo
    K_{q} \big( \cP_{G}[\pi\inv S] \big)
  \end{equation}
  which are natural in~$S$.

  Now observe that the category $\cF_{N}$ of finitely generated
  $N$-graded free $R_{N}$-modules is equivalent to
  $\cF_{G}[\pi\inv(s)]$, via the functor that takes $\shift n R_{N}
  \in \cF_{N}$ to the $G$-graded free $R$-module $\shift {\sigma(s)n}
  R$. This is a bijection on objects as $\pi\inv(s) =
  \sigma(s)N$. Morphisms $\shift {n_{1}} R_{N} \rTo \shift {n_{2}}
  R_{N}$ are in bijective correspondence with elements of $(\shift
  {n_{2}} R_{N})_{n_{1}} = (R_{N})_{n_{2}\inv n_{1}}$, and morphisms
  $\shift {\sigma(s)n_{1}} R \rTo \shift {\sigma(s)n_{2}} R$ are in
  bijective correspondence with elements of $\shift {\sigma(s)n_{2}} R
  _{\sigma(s)n_{1}} = R_{n_{2}\inv n_{1}}$, which is the same set.
  From Lemma~\ref{lem:equivalent} we conclude that $\cP_{G}[\pi\inv
  (s)]$ is equivalent to the category $\cP_{N}$ of finitely generated
  $N$-graded projective $R_{N}$-modules.

  Combining this with the isomorphism~\eqref{eq:iso_for_S}, and
  passing to the limit with respect to~$S$, we obtain an isomorphism
  of $K$-groups as stated in the Theorem.
\end{proof}

\begin{remark}
  Define a group $N \tilde \times_{\sigma} H$ which has $N \times H$
  as underlying set, has neutral element $\big(\sigma(1)^{-1},1
  \big)$, and has multiplication law
  \begin{displaymath}
    (n_{1}, h_{1}) \cdot (n_{2}, h_{2}) = \big(
    \sigma(h_{1}h_{2})^{-1} \sigma(h_{1}) \sigma(h_{2})
    n_{1}^{\sigma(h_{2})}n_{2}, h_{1}h_{2} \big) \ 
  \end{displaymath}
  where $a^{b} = b^{-1}ab$. The map $\mu(n,h) = \sigma(h)n$ is a group
  isomorphism $\mu \colon N \tilde \times H \rTo G$ with inverse
  $\alpha(g) = \big( \sigma(\pi(g))^{-1}g ,\pi(g)\big)$.  The
  left-hand side of~\eqref{eq:1} has an $N \tilde \times_{\sigma}
  H$-module structure described by saying that the action of
  $(n_{1},h_{1})$ sends the module $\shift {n_{2}} R_{N}$ in the
  $h_{2}$-summand to the module $\shift {\bar n} R_{N}$ in the
  $h_{1}h_{2}$-summand, where $\bar n = \sigma(h_{1}h_{2})^{-1}
  \sigma(h_{1}) \sigma(h_{2}) n_{1}^{\sigma(h_{2})}n_{2}$. The
  isomorphism then becomes a $\mu$-semilinear map in the sense that
  $\Psi \big( (n,h) \cdot x) = \mu(n,h) \cdot \Psi(x)$.
\end{remark}

Theorem~\ref{thm:syst_K} and
Proposition~\ref{prop:strong_systematic_equivalence} (applied to $Q =
N$ and $K = R_{N}$) yield:

\begin{corollary}
  \label{cor:Au_Walker}
  Suppose the $G$-systematic ring~$R$ has support in~$G^{+}$. Suppose
  further that the $N$-systematic ring~$R_{N}$ is strongly
  systematic. Then there are isomorphisms of $K$-groups
  \begin{displaymath}
    \bigoplus_{H} K_{n} (R_{1}) \iso K^{G\mathrm{\text{-}gr}}_{n} (R)
    \ .\tag*{\qedsymbol}
  \end{displaymath}
\end{corollary}

\section{Equivariant $K$-theory of affine toric schemes}

Corollary~\ref{cor:Au_Walker} generalises a result on graded
$K$-theory obtained by \textsc{Au}, \textsc{Huang} and \textsc{Walker}
\cite[Theorem~1]{MR2494374}. In their set-up, $G$ is an
\textsc{abel}ian group, $A \subseteq G$ a sub-monoid, $R = B[A]$ the
monoid ring over a commutative ground ring~$B$, and $N$ is the group
of invertible elements in~$A$. The partial order on $H = G/N$ is
defined by
\begin{displaymath}
  h' \geq h \quad \Leftrightarrow \quad g'g\inv \in A \ ,
\end{displaymath}
where $g,g' \in G$ are representatives of the cosets $h = gN$ and
$h'=g'N$. (That is, $A = G^{+}$ in the notation of
Theorem~\ref{thm:syst_K}.) Corollary~\ref{cor:Au_Walker} then reduces
to the cited result
\begin{displaymath}
  \bZ[G/N] \tensor_{\bZ} K_{n} (B)  \iso \bigoplus_{G/N} K_{n} (B) \iso
  K^{G\mathrm{\text{-}gr}}_{n} \big( B[A] \big) \ .
\end{displaymath}
As explained in {\it loc.cit.}, the result can be further specialised
and translated into the language of affine toric schemes. We will look
at a slight generalisation. So suppose in addition to the above that
$G \iso \bZ^{r}$ is an $r$-dimensional lattice, and that $A =
\sigma^{\vee} \cap G$ for some rational polyhedral $r$-dimensional
cone $\sigma^{\vee} \subseteq G \tensor_{\bZ} \bR \iso \bR^{r}$. Then
$R$ is the coordinate ring of the (possibly singular) affine toric
$B$-scheme $X = \mathrm{Spec} R$.

Suppose further that $R'$ is a commutative $G$-graded $R$-algebra,
that is, a commutative $G$-graded ring equipped with a
degree-preserving ring homomorphism $\nu \colon R \rTo
R'$. Geometrically this corresponds to a morphism of affine schemes
\begin{displaymath}
X' = \mathrm{Spec} R' \rTo \mathrm{Spec} R = X\ ;
\end{displaymath}
due to the presence of $G$-gradings, the $r$-dimensional algebraic
torus $T = \mathrm{Spec} B[G]$ acts on both source and target, and the
morphism is $T$-equivariant.

Note now that since $R_{N}$ is strongly graded so is $R'_{N}$; indeed,
for any $n,n' \in G$ we have $R'_{n} \tensor R'_{n'} \supseteq R'_{n}
\tensor \nu (R_{n'})$, and the restriction
\begin{displaymath}
   R'_{n} \tensor \nu (R_{n'}) \rTo  R'_{nn'}
\end{displaymath}
of the multiplication map in~$R'$ is surjective since any module over
a strongly graded ring is itself strongly graded
(Lemma~\ref{lem:systematic_Dade} or \cite[Theorem~2.8]{MR593823},
applied to the right $R_{N}$-module~$R'_{N}$). As $T$-equivariant
vector bundles on~$\mathrm{Spec} R'$ correspond to finitely generated
$G$-graded projective $R'$-modules, Corollary~\ref{cor:Au_Walker}
yields the following generalisation of \cite[Theorem~4]{MR2494374}:

\begin{theorem}
  If in the situation set out above the $G$-graded $R$-algebra $R'$
  has support in~$G^{+}$, there are isomorphisms of \textsc{abel}ian
  groups
  \begin{displaymath}
    K^{T}_{n} (X') \iso \bigoplus_{G/N} K_{n} (R'_{1}) \iso \bZ[G/N]
    \tensor_{\bZ} K_{n} (R'_{1}) \ . \tag*{\qedsymbol}
  \end{displaymath}
\end{theorem}

\end{document}